\documentclass[10pt]{amsart}
\usepackage{enumerate,amsmath,amssymb,latexsym,
amsfonts, amsthm, amscd, MnSymbol}

\theoremstyle{plain}

\newtheorem{theorem}{Theorem}

\numberwithin{equation}{section}

\newcommand{\ra}{\rightarrow}

\newcommand{\OO}{\Omega} 
\newcommand{\oo}{\omega} 
\newcommand{\ind}{{\rm ind}}

\addtocounter{section}{-1}

\begin{document}

\title {Index theory for partial-bijections}

\date{}

\author[P.L. Robinson]{P.L. Robinson}

\address{Department of Mathematics \\ University of Florida \\ Gainesville FL 32611  USA }

\email[]{paulr@ufl.edu}

 \keywords{}

\begin{abstract}

We offer streamlined proofs of fundamental theorems regarding the index theory for partial self-maps of an infinite set that are bijective between cofinite subsets. 

\end{abstract}

\maketitle

\medbreak

\section{Introduction} 

In a recent paper [1] we called the self-map $f$ of the infinite set $\OO$ a {\it near-bijection} precisely when $f$ restricts to a true bijection from a cofinite subset $A \subseteq \OO$ to a cofinite subset $B \subseteq \OO$. Along with the range $f(\OO)$ of $f$ we introduced its `monoset' 
$$\OO_f = \{ \oo \in \OO : \overleftarrow{f} (f(\oo)) = \{ \oo \} \};$$ 
in these terms, $f$ is a near-bijection precisely when $f(\OO)$ and $\OO_f$ are cofinite. The {\it index} of the near-bijection $f$ is then defined by 
$${\ind} (f) = |\OO_f'| - |f(\OO_f')| - |f(\OO)'| \in \mathbb{Z}$$
where if $C \subseteq \OO$ then $|C|$ is its cardinality and $C' = \OO \setminus C$ is its complement. We showed in [1] that the index is insensitive to changes on a finite set and that ${\ind} (f)$ is zero precisely when $f$ differs from a bijection on a finite set. We also showed that when near-bijections that differ on a finite set are identified, their equivalence classes constitute a group on which the $\mathbb{Z}$-valued index is a homomorphism. 

\medbreak 

In [1] considerable effort was devoted to the careful handling of $\OO_f'$ and $f(\OO)'$: for example, when the value of a near-bijection $f$ is changed at $\oo \in \OO_f'$  it is important to know whether the number of points at which $f(\oo)$ was formerly the value is two or is greater than two; it is also important to know whether the new value of $f$ at $\oo$ was or was not formerly a value of $f$. These circumstances cause technical complications: for example, in the verification that the index is insensitive to changes on a finite set and in the verification that the index is a homomorphism. Our primary purpose in this paper is to reformulate the notion of near-bijection in a way that circumvents these complications and facilitates streamlined proofs of the fundamental results.  

\medbreak 

\section{Index theory} 

\medbreak 

Let $\OO$ be an infinite set. 

\medbreak 

{\bf Definition:} A {\it partial-bijection} is a (true) bijection $f : A_f \ra B_f$ from a cofinite subset $A_f \subseteq \OO$ to a cofinite subset $B_f \subseteq \OO$. 

\medbreak 

This is our reformulation of the notion of near-bijection: as $f(\OO)'$ and (especially) $\OO_f'$ were the source of complications in [1] we simply eliminate them; much of the focus in [1] was on properties defined only up to changes on finite sets, so this reformulation is eminently reasonable. More strictly, we should perhaps speak of a bijective partial self-map; but the convenient abuse `partial-bijection' is also reasonable. 

\medbreak 

{\bf Definition:} The {\it index} of the partial-bijection $f$ is defined by 
$${\ind} (f) = |A_f'| - |B_f'| \in \mathbb{Z}.$$

\medbreak 

We should of course verify that this notion of index agrees with the notion in [1]. To this end, let $f: \OO \ra \OO$ be a near-bijection in the sense of [1]: that is, a map for which the complements $\OO_f'$ (see the Introduction) and $f(\OO)'$ are finite. Restricting $f$ (but using the same symbol for convenience) yields a partial-bijection $f: \OO_f \ra f(\OO_f)$: in fact, 
$$\OO \setminus f(\OO_f) = (\OO \setminus f(\OO)) \cup (f(\OO) \setminus f(\OO_f))$$
where $f(\OO) \setminus f(\OO_f) = f(\OO \setminus \OO_f)$ by definition of $\OO_f$; accordingly, 
$$|f(\OO_f)'| = |f(\OO)'| + |f(\OO_f')|.$$
It follows that the index of $f: \OO \ra \OO$ as defined in [1]  is 
$$|\OO_f'| - |f(\OO_f')| - |f(\OO)'| = |\OO_f'| - |f(\OO_f)'|$$
and so agrees with the index of $f: \OO_f \ra f(\OO_f)$ as presently defined. 

\medbreak 

With the current definitions, the following is immediate. 

\medbreak 

\begin{theorem} \label{zero}
The partial-bijection $f$ extends to a true bijection $\OO \ra \OO$ precisely when ${\ind} (f)$ vanishes. 
\end{theorem} 

\begin{proof} 
The bijection $f: A_f \ra B_f$ extends to a bijection from $\OO$ to itself precisely when the finite complements $A_f'$ and $B_f'$ have the same cardinality. 
\end{proof} 

\medbreak 

In [1] we identified near-bijections when they differed on a finite set. Here, the corresponding identification results from the following definition. 

\medbreak 

{\bf Definition:}  The partial-bijections $f$ and $g$ are {\it almost equal} (notation: $f \equiv g$) precisely when $f$ and $g$ agree on a cofinite subset of $A_f \cap A_g$. 

\medbreak 

It is readily verified that almost equality is an equivalence relation; transitivity would fail were we simply to insist that $f$ and $g$ agree on their common domain $A_f \cap A_g$. 

\medbreak

As expected, the index is insensitive to changes on a finite set. 

\medbreak 

\begin{theorem} \label{ind}
Let $f$ and $g$ be partial-bijections. If $f \equiv g$ then $\ind (f) = \ind (g)$. 
\end{theorem} 

\begin{proof} 
Let $f$ and $g$ agree on the cofinite set $E \subseteq A_f \cap A_g$ and write $e : E \ra e(E)$  for the common restriction $f|_E = g|_E$. Now $f$ restricts to a bijection $A_f \setminus E \ra B_f \setminus e(E)$ where 
$$A_f \setminus E = A_f \cap E' = E' \setminus A_f'$$
and 
$$B_f \setminus e(E) = B_f \cap e(E)' = e(E)' \setminus B_f'$$
whence 
$$|E'| - |A_f'| = |e(E)'| - |B_f'|$$
and therefore $$\ind (f) = |A_f'| - |B_f'| = |E'| - |e(E)'| = \ind (e).$$
The symmetric observation that $\ind (e)$ equals $\ind (g)$ completes the proof. 
\end{proof} 

\medbreak 

{\bf Remark:} Comparison with the proof of the corresponding result (Theorem 19) in [1] amply demonstrates the virtue of the approach adopted here. 

\medbreak 

Let us now consider the composition of partial-bijections. The natural approach to composition of partial maps suggests the following.

\medbreak 

{\bf Definition:} The {\it composite} of the partial-bijections $f$ and (then) $g$ is the map $g \circ f : A_{g \circ f} \ra B_{g \circ f}$ with domain 
$$A_{g \circ f} := \overleftarrow{f} (B_f \cap A_g)$$
with codomain 
$$B_{g \circ f} := \overrightarrow{g} (B_f \cap A_g)$$
and with rule 
$$(\forall \oo \in \OO) \; \; \; (g \circ f) (\oo) := g(f(\oo)).$$

\medbreak 

\begin{theorem} \label{comp}
If $f$ and $g$ are partial-bijections, then so is $g \circ f$. 
\end{theorem} 

\begin{proof} 
That $g \circ f: A_{g \circ f} \ra B_{g \circ f}$ is a bijection is clear. To see that $A_{g \circ f}$ is cofinite, observe that 
$$A_{g \circ f}' = \overleftarrow{f} (B_f \cap A_g)' = (\OO \setminus A_f) \cup (A_f \setminus \overleftarrow{f} (B_f \cap A_g))$$
where $|\OO \setminus A_f| = |A_f'|$ and where (as $f$ is a bijection from $A_f$ to $B_f$)
$$|A_f \setminus \overleftarrow{f} (B_f \cap A_g)| = |B_f \setminus (B_f \cap A_g)| = |B_f \setminus A_g| = |A_g' \setminus B_f'|.$$
Thus 
$$|\overleftarrow{f} (B_f \cap A_g)'| = |A_f'| + |A_g' \setminus B_f'|$$
is finite, as likewise is 
$$| \overrightarrow{g} (B_f \cap A_g)'| = |B_g'| + |B_f' \setminus A_g'|.$$
\end{proof} 

\medbreak 

We now propose to prove that when partial-bijections are composed their indices add. For this purpose, it is convenient first to record the following triviality. 

\medbreak 

\begin{theorem} \label{triv}
If $X$ and $Y$ are finite sets then $|X \setminus Y| + |Y| = |Y \setminus X| + |X|.$
\end{theorem}

\begin{proof}
Each side of the equation is precisely $|X \cup Y|$. 
\end{proof} 

\medbreak 

Verification of our claim regarding the index of a composite is now quite straightforward. 

\begin{theorem} \label{sum}
If $f$ and $g$ are partial-bijections then $\ind (g \circ f) = \ind (g) + \ind (f).$
\end{theorem} 

\begin{proof} 
We continue from the close of the proof for Theorem \ref{comp}. Thus 
$$\ind (g \circ f) = |\overleftarrow{f} (B_f \cap A_g)'| - | \overrightarrow{g} (B_f \cap A_g)'| = |A_f'| + |A_g' \setminus B_f'| - |B_g'| - |B_f' \setminus A_g'|$$
while 
$$\ind (g) + \ind (f) = |A_g'| - |B_g'| + |A_f'| - |B_f'|$$
whence 
$$\ind (g \circ f) - \ind (g) - \ind (f) =  |A_g' \setminus B_f'| + |B_f'| - |B_f' \setminus A_g'| - |A_g'|$$
and an application of Theorem \ref{triv} with $X = A_g'$ and $Y = B_f'$ ends the argument. 
\end{proof}

\medbreak 

{\bf Remark:} In [1] the corresponding result is Theorem 27; once again, comparison highlights the virtue of the approach taken in the present paper. 

\medbreak 

Let $f: A_f \ra B_f$ be a partial bijection: as a bijection, $f$ has an inverse map $f^{-1} : B_f \ra A_f$ which is also a partial-bijection; the composites $f^{-1} \circ f = {\rm Id}_{A_f}$ and $f \circ f^{-1} = {\rm Id}_{B_f}$ imply that $f^{-1} \circ f \equiv {\rm Id}_{\OO} \equiv f \circ f^{-1}$. As a companion to the last theorem, we have the next. 

\medbreak 

\begin{theorem} \label{neg}
If $f$ is a partial-bijection then $\ind (f^{-1}) = - \ind (f).$
\end{theorem} 

\begin{proof} 
Immediate: passage from $f$ to its inverse switches $A_f$ and $B_f$. 
\end{proof} 

\medbreak 

The permutations of $\OO$ make up the symmetric group $S_{\OO}$; Theorem \ref{zero} says that the partial-bijections having index zero are exactly the restrictions of these permutations to cofinite subsets of $\OO$. It is clear that composing a partial-bijection with a permutation (on either side, left or right) does not affect the index: indeed, if $f$ is a partial-bijection and $\pi$ is a permutation, then $A_{\pi \circ f}' = A_f'$ and $B_{\pi \circ f}' = \overrightarrow{\pi} (B_f)' = \overrightarrow{\pi} (B_f')$ so that 
$$\ind (\pi \circ f) = |A_{\pi \circ f}'| - |B_{\pi \circ f}'| = |A_f'| - |\pi(B_f')| = |A_f'| - |B_f'| = \ind (f)$$
while $A_{f \circ \pi}' = \overleftarrow{\pi} (A_f)' = \overleftarrow{\pi} (A_f')$ and $B_{f \circ \pi}' = B_f'$ so that 
$$\ind (f \circ \pi) = |A_{f \circ \pi}'| - |B_{f \circ \pi}'| = | \overleftarrow{\pi} (A_f')| - |B_f'| = |A_f'| - |B_f'| = \ind (f).$$

\medbreak 

In fact, this essentially covers all cases of equal index: any two partial-bijections having the same index are related by permutations in this way, up to almost equality. 

\medbreak 

\begin{theorem} \label{perm}
Let $f$ and $g$ be partial-bijections. If $\ind (f) = \ind (g)$ then there exist permutations $\lambda \in S_{\OO}$ and $\rho \in S_{\OO}$ such that $\lambda \circ f \equiv g \equiv f \circ \rho.$
\end{theorem} 

\begin{proof} 
The composite partial-bijection $g$ after $f^{-1}$ is a true bijection 
$$g \circ f^{-1}: \overrightarrow{f} (A_g) \ra \overrightarrow{g} (A_f). $$
Theorem \ref{sum} and Theorem \ref{neg} show that $\ind (g \circ f^{-1}) = \ind (g) - \ind (f) = 0$ and then Theorem \ref{zero} shows that the partial-bijection $g \circ f^{-1}$ extends to a permutation $\lambda$ of $\OO$; the almost equality $\lambda \circ f \equiv g$ is clear. Similarly, the composite partial-bijection 
$$f^{-1} \circ g: \overleftarrow{g} (B_f) \ra \overleftarrow{f} (B_g)$$
has vanishing index and extends to a permutation $\rho$ of $\OO$ such that $g \equiv f \circ \rho$. 
\end{proof} 

\medbreak 

{\bf Remark:} This demonstrates quite strikingly the virtue of the present approach when dealing with matters that allow indeterminacy on finite sets. In [1] the corresponding result is a combination of Theorem 21, Theorem 22 and Theorem 23; there, complications involving range and monoset necessitate the separate handling of $\lambda$ and $\rho$ as well as the separate handling of negative index and positive index. 

\medbreak 

As in [1] it is of interest to view these results from a group-theoretic perspective. Almost equality defines an equivalence relation on the set of all partial-bijections of $\OO$; we denote by $\mathbb{G}_{\OO}$ the set comprising all such $\equiv$-classes, denoting the $\equiv$-class of $f$ by $[f]$ as usual. 

\begin{theorem} 
Let $f_1, f_2, g_1, g_2$ be partial-bijections. If $f_1 \equiv f_2$ and $g_1 \equiv g_2$ then $g_1 \circ f_1 \equiv g_2 \circ f_2.$ 
\end{theorem} 

\begin{proof} 
Note from Theorem \ref{comp} that $g_1 \circ f_1$ and $g_2 \circ f_2$ are partial-bijections. Let $F \subseteq A_{f_1} \cap A_{f_2}$ and $G \subseteq A_{g_1} \cap A_{g_2}$ be cofinite sets on which $f_1|_F = f_2|_F =: f$ and $g_1|_G = g_2|_G =: g$. Verification that  $g_1 \circ f_1$ and $g_2 \circ f_2$ agree on $\overleftarrow{f} (G)$ is immediate; verification that $\overleftarrow{f} (G)$ is cofinite presents no difficulties. 
\end{proof} 

\medbreak 

It follows that composition descends to a well-defined (associative) binary operation on $\mathbb{G}_{\OO}$; this makes $\mathbb{G}_{\OO}$ into a group, the inverse of $[f]$ being $[f^{-1}]$. Theorem \ref{ind} guarantees that the index map $\ind$ descends to a well-defined map 
$${\rm Ind} : \mathbb{G}_{\OO} \ra \mathbb{Z}$$
which is a group homomorphism by Theorem \ref{sum}. By Theorem \ref{zero}, the kernel $\mathbb{S}_{\OO}$ of ${\rm Ind}$ comprises precisely all $\equiv$-classes containing permutations. The image of ${\rm Ind}$ is of course $\mathbb{Z}$: note that if $\oo_0 \in \OO$ then any bijection $u: \OO \ra \OO \setminus \{ \oo_0 \}$ has index $- 1$ and if $n \in \mathbb{Z}$ then $[u]^n$ has index $- n$. The cosets of $\mathbb{S}_{\OO} \leqclosed \mathbb{G}_{\OO}$ are labelled by ${\rm Ind}$: this is clear from the fundamental isomorphism theorem, but is also explicit in Theorem \ref{perm} and the discussion leading up to it. 

\medbreak 

Thus, we have constructed a short exact sequence of groups 
$${\rm Id} \ra \mathbb{S}_{\OO} \ra \mathbb{G}_{\OO} \ra \mathbb{Z} \ra 0.$$
This sequence splits, as $\mathbb{Z}$ is infinite cyclic: with $u$ as above, a splitting homomorphism is 
$$\mathbb{Z} \ra \mathbb{G}_{\OO} : n \mapsto [u]^{- n}.$$

\medbreak 

In summary, the approach taken in this paper, based on partial-bijections in place of near-bijections, offers a significantly streamlined route to those results of [1] pertaining to properties that are unaffected by changes on finite subsets of $\OO$; in particular, it is well-suited to handling the group $\mathbb{G}_{\OO}$ and the index. 
\bigbreak

\begin{center} 
{\small R}{\footnotesize EFERENCES}
\end{center} 
\medbreak 

[1] P.L. Robinson, {\it Fredholm theory for cofinite sets}, arXiv 1509.08039 (2015). 

\medbreak

\end{document}